\documentclass{amsart}

\title[Existence of Infinitely Many Minimal Hypersurfaces]{Existence of Infinitely Many Minimal Hypersurfaces in Higher-dimensional Closed Manifolds with Generic Metrics}
\author{Yangyang Li}
\address{Department of Mathematics, Princeton University, Princeton, NJ 08544}
\email{yl15@math.princeton.edu}
\thanks{The author is partially supported by NSF-DMS-1811840.}

\usepackage{bm}
\usepackage{amsfonts,amssymb}
\usepackage{mdframed}
\usepackage{amsmath}
\usepackage{amsthm}
\usepackage{mathrsfs}
\usepackage{graphicx}
\usepackage{tikz}
\usepackage{galois}
\usepackage{soul}
\usepackage{xcolor}
\usepackage[all]{xy}
\usepackage[colorlinks, linkcolor = blue, anchorcolor = blue, citecolor = blue]{hyperref}

\soulregister\cite7 
\soulregister\citep7 
\soulregister\citet7 
\soulregister\ref7 
\soulregister\pageref7 

\newtheorem{definition}{Definition}[section]

\newtheorem{proposition}{Proposition}[section]
\newtheorem{corollary}{Corollary}[section]
\newtheorem{theorem}{Theorem}[section]

\newtheorem{remark}{Remark}[section]
\newtheorem{lemma}{Lemma}[section]

\newtheorem{conjecture}{Conjecture}[section]

\DeclareRobustCommand*{\myhl}[1]{%
    #1}
\newcommand{\set}[1]{\{#1\}}

\begin{document}
\bibliographystyle{abbrvalpha}
\begin{abstract}
    In this paper, we show that a closed manifold $M^{n+1} (n \geq 7)$ endowed with a $C^\infty$-generic (Baire sense) metric contains infinitely many singular minimal hypersurfaces with optimal regularity. Moreover, for $2 \leq n \leq 6$, our argument also implies the denseness of the minimal hypersurfaces realizing min-max widths for generic metrics. This partially supports equidistribution of the minimal hypersurfaces realizing min-max widths conjectured by F.C. Marques, A. Neves and A. Song in \cite{marques_equidistribution_2019}.
\end{abstract}

\maketitle
\section{Introduction}

    In Riemannian geometry, the existence and regularity of minimal hypersurfaces is one of the central problems. In 1982, motivated by the existence results in $(n+1)$-dimensional closed manifolds by G.D. Birkhoff (\cite{birkhoff_dynamical_1917}, $n=1$), J. Pitts (\cite{pitts_existence_1981}, $2\leq n \leq 5$) and R. Schoen and L. Simon (\cite{schoen_regularity_1981}, $n \geq 6$), S.-T. Yau proposed the conjecture of existence of infinitely many minimal surfaces in \myhl{closed $3$-dimensional Riemannian manifolds.}

    \begin{conjecture}[Yau's conjecture, \cite{10.2307/j.ctt1bd6kkq.37}]
        Any closed three-dimensional manifold must contain an infinite number of immersed minimal surfaces.
    \end{conjecture}

    In \cite{irie_density_2018}, K. Irie, F.C. Marques and A. Neves, using the Weyl law \cite{liokumovich_weyl_2018} for volume spectra by Y. Liokumovich and the last two named authors, proved a stronger version of Yau's conjecture in the generic case.

    \begin{theorem}[Density of minimal hypersurfaces in the generic case, \cite{irie_density_2018}]
        Let $M^{n+1}$ be a closed manifold of dimension $n+1$, with $3\leq n+1 \leq 7$. Then for a $C^\infty$-generic Riemannian metric $g$ on $M$, the union of all closed, smooth, embedded minimal hypersurfaces is dense.
    \end{theorem}

    Later, in \cite{marques_equidistribution_2019}, F.C. Marques, A. Neves and A. Song gave a quantitative description of the density, i.e., the equidistribution of a sequence of minimal hypersurfaces under the same condition.

    The Yau's conjecture for $2 \leq n \leq 6$ for general $C^\infty$ metrics was finally resolved by A. Song \cite{song_existence_2018} using the methods developed by F.C. Marques and A. Neves in \cite{marques_existence_2017}.

    Recently, X. Zhou \cite{zhou_multiplicity_2019} confirmed Marques-Neves multiplicity one conjecture for bumpy metrics, which combined with work of Marques-Neves \cite{marques_morse_2018-1} on the Morse index leads to:

    \begin{theorem}[\cite{marques_morse_2018-1}, Theorem 8.4]
        Let $g$ be a $C^\infty$-generic (bumpy) metric on a closed manifold $M^{n+1}$, $3 \leq (n+1) \leq 7$. For each $k\in \mathbb{N}$, there exists a smooth, closed, embedded, multiplicity one, two-sided, minimal hypersurface $\Sigma_k$ such that
        \begin{equation}
            \omega_k(M, g) = \mathrm{area}_g(\Sigma_k)\myhl{,} \quad \mathrm{index}(\Sigma_k) =k,
        \end{equation}
        and
        \begin{equation}
            \lim_{k\rightarrow \infty}\frac{\mathrm{area}_g(\Sigma_k)}{k^{\frac{1}{n+1}}} = a(n) \mathrm{vol}(M, g)^{\frac{n}{n+1}},
        \end{equation}
        where \myhl{$\omega_k(M,g)$ is the volume spectrum, more precisely, the min-max $k$-width,} and $a(n) > 0$ is a dimensional constant in Weyl law.
    \end{theorem}

    Note that most of the results above were obtained in the Almgren-Pitts min-max setting (Zhou's result on the multiplicity one conjecture was based on a new regularization of the area functional in Cacciopoli min-max setting developed by him and J. Zhu \cite{zhou_existence_2018}). In the Allen-Cahn min-max setting, P. Gaspar and M.A.M. Guaraco \cite{gaspar_weyl_2019} and O. Chodosh and C. Mantoulidis \cite{chodosh_minimal_2018}($n=2$) gave similar results. In particular, O. Chodosh and C. Mantoulidis proved the multiplicity one conjecture in $3$-manifolds before Zhou's result.

    Most of the results above rely on two important ingredients, the upper bounds for Morse index and the denseness of bumpy metrics (\cite{white_space_1991}). However, they could not be easily generalized in higher dimension ($n \geq 7$) directly from the literature above, especially in the Almgren-Pitts setting. Thus, Yau's conjecture is still left open in higher-dimensional closed manifolds.

    In this paper, we will confirm Yau's conjecture in higher dimension ($n\geq 7$) for a closed manifold $M$ with a $C^\infty$-generic metric $g$. Due to the existence of nontrivial singularities in \myhl{area-minimizing currents} in higher dimension, we can at best expect that the minimal hypersurfaces have optimal regularity, i.e., $\mathrm{codim}(\mathrm{Sing}) \geq 7$.

    \begin{theorem}[Main Theorem]
        Given a closed manifold $M^{n+1}(n \geq 7)$, there exists a (Baire sense) generic subset of $C^\infty$ metrics such that $M$ endowed with any one of those metrics contains infinitely many singular minimal hypersurfaces with optimal regularity.
    \end{theorem}

    Here is the outline of the proof of our main theorem.

    First, we establish the compactness for the Almgren-Pitts Realization $\mathcal{APR}(\Pi)$ \myhl{of the min-max width} of an $m$-parameter homotopy family $\Pi$, where $\mathcal{APR}(\Pi)$ is a nonempty subset of minimal hypersurfaces with optimal regularity, volume $\mathbf{L}(\Pi)$ and certain \myhl{stability property}.

    Then, by showing that the $p$-width could be achieved by \myhl{the min-max width of} an $m$-parameter homotopy family $\mathcal{P}_{p,m}$, we also have the compactness for the Almgren-Pitts Realization $\mathcal{APR}_p$ of $p$-width.

    Finally, the proof of the main theorem follows the idea of Proposition 3.1 in \cite{irie_density_2018} with a tricky modification. Roughly speaking, to overcome the difficulty of the lack of bumpy metrics, we will apply the compactness results to obtain the openness of ``good'' metrics and use $\mathcal{M}_f$ (the set of metrics where Yau's conjecture fails) to be the starting point of metric perturbation. If we divide the metric space into two parts $\mathcal{O} = \mathrm{Int}(\overline{\mathcal{M}_f})$ and $\mathcal{O}^c$, then on the one hand, in $\mathcal{O}$, with two ingredients mentioned above, we can show that $\mathcal{M}_f\cap \mathcal{O}$ is meagre. On the other hand, by definition, it is clear that $\mathcal{M}_f \cap \mathcal{O}^c$ is nowhere dense. \myhl{In summary,} $\mathcal{M}_f$ is meagre.

    In addition, together with White's structure theorem on minimal submanifolds \cite{white_bumpy_2017}, our proof above also implies the generic denseness of min-max minimal hypersurfaces.

    \begin{theorem}[a.k.a. Corollary \ref{cor:end}]
        For a closed manifold $M^{n+1}(2 \leq n \leq 6)$ with a $C^\infty$-generic metric $g$, min-max minimal hypersurfaces are dense.
    \end{theorem}

\section*{Acknowledgements}
    
    I am grateful to my advisor Fernando Cod\'a Marques for his constant support and his patient guidance in my study of Almgren-Pitts min-max theory. I would also thank Xin Zhou for inspiring discussions and his pointing out reference \cite{xu_p;m-width_2016}. Finally, I would like to thank the referee for their careful review.

\section{Preliminaries}
    
    \subsection{Basic Notations in the Almgren-Pitts Min-max Theory}

        In this paper, we will use $\mathcal{Z}_n(M^{n+1}, \mathbb{Z}_2)$ to denote the space of modulo two $n$-dimensional flat cycles \myhl{endowed with the flat topology}. In Almgren's thesis \cite{almgren_homotopy_1962}, he gave a natural isomorphism
        \begin{equation}
            \pi_k(\mathcal{Z}_n(M^{n+1}, \mathbb{Z}_2), 0) = H_{n+k}(M^{n+1}, \mathbb{Z}_2),
        \end{equation}
        and later, it was shown that $\mathcal{Z}_n(M^{n+1}, \mathbb{Z}_2)$ is weakly homotopic to $\mathbb{RP}^\infty$ (\cite{marques_topology_2016}, Section 4). We denote the generator of $H^1(\mathcal{Z}_n(M^{n+1}, \mathbb{Z}_2); \mathbb{Z}_2)=\mathbb{Z}_2$ by $\bar\lambda$.
        
        In the cycles space, we can define $\mathbf{M}$ to be the mass norm or mass functional (See \cite{federer_geometric_1996}, Section 4.2). The nontrivial topology of $\mathcal{Z}_n(M^{n+1}, \mathbb{Z}_2)$ indicates that the min-max theory for the mass functional can be developed in it\myhl{.}

        Let $X$ be a finite dimensional simplicial complex and we call a map $\Phi$ a \textbf{$p$-sweepout} if it is a continuous map $\Phi: X \rightarrow \mathcal{Z}_n(M^{n+1}, \mathbb{Z}_2)$ such that $\Phi^*(\bar\lambda^p) \neq 0$. The \textbf{$p$-admissible set} $\mathcal{P}_p = \mathcal{P}_p(M, g)$ is the set of all $p$-sweepouts $\Phi$ that have \textbf{no mass concentration} (See \cite{marques_min-max_2014} Section 4), i.e.
            \begin{equation}
                \myhl{\lim_{r\rightarrow 0}\sup}\set{\mathbf{M}(\Phi(x)\cap B_r(y)): x \in X, y \in M} = 0.
            \end{equation}

        Now the \textbf{$p$-width} $\omega_p(M,g)$ can be defined by $\inf_{\Phi \in \mathcal{P}_p}\sup \set{\mathbf{M}(\Phi(x)): x\in \mathrm{dmn}({}\Phi)}$ (\cite{gromov_isoperimetry_2003}, \cite{guth_minimax_2009}, \cite{marques_existence_2017}, \cite{liokumovich_weyl_2018}). 

        We also define a \textbf{min-max sequence} $S = \set{\Phi_i}_{i\in \mathbb{N}}$ to be a sequence in $\mathcal{P}_p$ satisfying 
        \begin{equation}
            \limsup_{i\rightarrow \infty}\sup_{x\in \myhl{X_i}} \mathbf{M}(\Phi_i(x)) = \omega_p.
        \end{equation}
        The critical set of $S$ is
        \begin{equation}
            \mathbf{C}(S) = \set{V| \|V\|(M) = \omega_p  \text{ and } V = \lim_j |\Phi_{i_j}(x_j)| \myhl{\text{ for some sequence } \set{x_j \in X_{i_j}}}}.
        \end{equation}

        Note that the domains of maps in $\mathcal{P}_p$ could be different. In order to get the compactness, we need to make some restriction on the domains. 

        \myhl{For convenience, we shall use some notions of cell complexes from \cite{pitts_existence_1981}.}

        \begin{definition}[\cite{pitts_existence_1981}, 4.1(1)]\quad
            \begin{itemize}
                \item \myhl{For $n \in \mathbb{N}^+$, let $I(1,n)$ denote the cell complex on the unit interval $I$ whose $1$-cells are the intervals $[0, 1 \cdot 3^{-n}], [1 \cdot 3^{-n}, 2 \cdot 3^{-n}], \cdots, [1 - 3^{-n}, 1]$, and whose $0$-cells are the endpoints $[0], [3^{-n}], [2\cdot 3^{-n}], \cdots, [1]$.} 
                \item \myhl{For $n, m \in \mathbb{N}^+$, $I(m, n) = I(1, n)^{m\otimes} = I(1, n) \otimes I(1,n) \otimes \cdots \otimes I(1,n)$ is the cell complex on $I^m$. In addition, $I(m, n)_0$ denote the set of all $0$-cells in $I(m, n)$.}
            \end{itemize}
        \end{definition}
    
        Now we can give a definition of an $m$-parameter homotopy family in general, which need not be a subset of some $\mathcal{P}_p$.

        \begin{definition}
            We call a subset $\Pi$ of continuous maps from finite dimensional simplicial complexes to $\mathcal{Z}_n(M^{n+1}, \mathbb{Z}_2)$ a \textbf{(continuous) $m$-parameter homotopy family} if the following properties hold.
            \begin{itemize}
                \item For any $\Phi \in \Pi$, $X = \mathrm{dmn}(\Phi)$ is a \myhl{subcomplex of $I(m, k)$ for some $k \in \mathbb{N}^+$} and $\Phi$ has no mass concentration.
                \item For any $\Phi \in \Pi$, \myhl{every continuous $\Phi': \mathrm{dmn}(\Phi) \rightarrow \mathcal{Z}_n(M^{n+1}, \mathbb{Z}_2)$ homotopic to $\Phi$ in the flat topology also lies in $\Pi$, provided that $\Phi'$ has no mass concentration.}
            \end{itemize}
        \end{definition}

        \begin{remark}\label{rmk:CTS}
            In Pitts' original \myhl{proof \cite{pitts_existence_1981}}, he considered \myhl{discrete sweepouts} in $[I^m, \mathcal{Z}_n(M; \mathbf{M}; \mathbb{Z}_2)]^\#$. Fortunately, due to the remarkable interpolation results by F.C. Marques and \myhl{A. Neves (\cite{marques_min-max_2014},} Section 13 and 14), the discrete settings and the continuous settings are interchangeable in some sense. Thus, the $m$-parameter homotopy family \myhl{defined} here will preserve most of the properties that the discrete homotopy families have.
        \end{remark}

        Similar to $p$-width and the min-max sequence for $p$-width, we can also define a min-max invariant for the $m$-parameter homotopy family $\Pi$,
        \begin{equation}
            \mathbf{L}(\Pi) = \inf_{\Phi \in \Pi}\sup_{x\in X} \set{\mathbf{M}(\Phi(x))},
        \end{equation}
        and a min-max sequence $S = \set{\Phi_i}\subset \Pi$,
        \begin{equation}
            \limsup_{i\rightarrow \infty}\sup_{x\in X_i = \mathrm{dmn}(\Phi_i)} \mathbf{M}(\Phi_i(x)) = \mathbf{L}(\Pi).
        \end{equation}
        Now the critical set of $S$ is
        \begin{equation}
            \mathbf{C}(S) = \set{V|\|V\|(M) = \mathbf{L}(\Pi) \text{ and } V = \lim_{j \rightarrow\infty} |\Phi_{i_j}(x_j)| \myhl{\text{ for some sequence } \set{x_j \in X_{i_j}}}}.
        \end{equation}

    \subsection{Singular Minimal Hypersurfaces with Optimal Regularity and Compactness with Stability Condition}

        \begin{definition}
            We call a varifold $V$ in $M^{n+1}$ a \textbf{singular minimal hypersurface with optimal regularity}, if it is an $n$-dimensional stationary integral varifold and its singular part has dimension $\mathrm{dim}(\mathrm{Sing}(V)) \leq n-7$.
        \end{definition}

        By \myhl{Allard compactness (See \cite{allard_first_1972}, \cite{simon_lectures_1984})}, a sequence $\set{V_j}$ of singular minimal hypersurfaces with optimal regularity with bounded volume (up to a subsequence) will converge to an $n$-dimensional stationary integral varifold in the varifold sense. However, without extra information of the sequence, we could not obtain further regularity of the limit varifold.

        If we further assume that $\set{V_j}$ is stable in some open subset $U$, then from the result of \cite{schoen_regularity_1981} (See also \cite{wickramasekera_general_2014}), we have the following compactness consequence.

        \begin{proposition}\label{prop:CPT0}
            Given a sequence of singular minimal hypersurfaces with optimal regularity $\set{V_j}$ in a smooth closed manifold $(M, g)$ with uniformly bounded volume and converging to an $n$-dimensional stationary integral varifold $V$ in the varifold sense, suppose that each $V_j$ is stable (See \myhl{the} definition in \cite{schoen_regularity_1981}) in an open subset $U$ of $M$, then $V$ is stable in $U$ and \myhl{of optimal regularity} in $U$.
        \end{proposition}
        \begin{remark}\label{rmk:CPTvary}
            This is also true when $V_j$ is defined on $(M, g_j)$ and $g_j$ converges to $g$ in $C^3$ (\cite{schoen_regularity_1981}, Theorem 2). 

            Recently, A. Dey \cite{dey_compactness_2019} generalized this result with the assumption that the $p$-th eigenvalue is uniformly bounded from below.
        \end{remark}

\section{Minimal Hypersurfaces from Almgren-Pitts Min-max Construction}

    In this section, we will prove some properties of singular minimal hypersurfaces \myhl{obtained} from Almgren-Pitts Min-max construction, especially realizations of $p$-widths and their compactness. 

    \subsection{Almgren-Pitts Realizations of $m$-parameter Homotopy Families and Their Compactness}

        \myhl{One of the novelties in Pitts' monograph is the application of the combinatorial arguments, Proposition 4.9 and Theorem 4.10 in \cite{pitts_existence_1981}, from which he could prove the regularity of varifolds in a critical set. Here, we shall show that these arguments could imply more properties of varifolds in a critical set.}

        \myhl{First, we adapt Proposition 4.9 in \cite{pitts_existence_1981} to our current setup and improve the constant slightly.}

        \myhl{In the following, $A(p, s, r)$ denotes an open annulus in $M$ centered at $p$ with inner radius $s$ and outer radius $r$, provided that $s < r$ and $r$ is no greater than the injective radius of $M$ at $p$. $\bar A(p, s, r)$ is simply its closure.}

        \begin{lemma}\label{lem:comb}
            \myhl{Suppose that $X$ is a subcomplex of some $I(m, n)$. If every cell $\sigma$ of $X$ is associated with a point $p_\sigma \in M$ and a finite set $A(\sigma) = \set{\bar A(p_\sigma, s_j, r_j)}_{j = 1, 2, \cdots, 5^m}$,
            where $r_j > s_j > 10r_{j+1} > 0$ ($r_1 < \mathrm{inj}_M/2$), then we can find a function defined on the set of cells of $X$ (denoted by $X(n)$),}
            \begin{equation}
                \alpha: X(n) \rightarrow \bigcup_{\sigma \in X(n)} A(\sigma),
            \end{equation}
            \myhl{such that $\alpha(\sigma) \in A(\sigma)$ and $\alpha(\sigma) \cap \alpha(\tau) = \emptyset$ whenver $\sigma \neq \tau$ and $\sigma, \tau < \gamma$ for some cell $\gamma \in X(n)$.}
        \end{lemma}

        \begin{proof}
            \myhl{We shall define $\alpha$ inductively. }

            \myhl{Let $\mathcal{D} \subset X(n)$ where $\alpha$ is already defined, and thus, at the beginning, $\mathcal{D} = \emptyset$.}

            \myhl{Whenever $\mathcal{D} \neq X(n)$, we can take an annulus with the smallest outer radius in $\bigcup_{\sigma \in X(n)\backslash \mathcal{D}} A(\sigma)$, say, $\bar A(p_{\sigma_0}, s, r) \in A(\sigma_0)$ for some $\sigma_0 \in X(n)\backslash \mathcal{D}$. Now,  $\sigma_0$ can be added into $\mathcal{D}$ as $\alpha(\sigma_0)$ is defined to be $\bar A(p_{\sigma_0}, s, r)$. Then, for each $\tau \in X(n)\backslash \mathcal{D}$ which is a face of some cell $\gamma \in X(n)$ with $\sigma_0 < \gamma$, we remove $\bar A(p_\tau, \tilde s, \tilde r) \in A(\tau)$ from the set $A(\tau)$ if $\bar A(p_\tau, \tilde s, \tilde r) \cap \bar A(p_{\sigma_0}, s, r) \neq \emptyset$. Since $\bar A(p_{\sigma_0}, s, r)$ is of the smallest outer radius, by the radius relations of the annuli in $A(\tau)$, it is easy to check that we at most remove one annulus from each $A(\tau)$.} 

            \myhl{Note that $X$ is a subcomplex of $I(m, n)$, so for any cell $\tau \in X(n)$, we can find at most $5^m - 1$ other cells such that any one of them and $\tau$ are both faces of some cell $\gamma \in X(n)$. In other words, there will be at most $5^m - 1$ annuli being removed from each $A(\tau)$ and therefore, $A(\tau)$ will never be empty. This, together with the fact that there are only finitely many cells of $X$, guarantees that the procedure above could give the definition of $\alpha$ on the whole $X(n)$ as desired.}
        \end{proof}

        \myhl{Now, we can state our main lemma on the properties of varifolds in a critical set.}

        \begin{lemma}\label{lemma: AM}
            \myhl{Suppose $S = \{\Phi_i\}$ is a pulled-tight critical sequence for a $m$-parameter homotopy family $\Pi$ with $\mathbf{L}(\Pi)>0$. Then there exists a varifold $V \in \mathbf{C}(S)$ such that for any $I_m = 5^m$ concentric annuli $\set{A(p, r_j - s_j, r_j + s_j)}$ with $r_j >0$, $s_j > 0$, $r_j - 2s_j > 10(r_{j+1} + 2 s_{j+1})$, $r_{I_m} - 2 s_{I_m}>0$ and $r_1 < \mathrm{inj}_M/2$, $V$ is almost minimizing (See Definitions 3.1 in \cite{pitts_existence_1981}) in at least one of the annuli. Therefore, by Theorem 3.3 in \cite{pitts_existence_1981}, $V$ has \textbf{property ($m$)}: }

            \myhl{For any $I_m$ concentric annuli $\set{A(p, r_j - s_j, r_j + s_j)}$ with $r_j >0$, $s_j > 0$, $r_j - 2s_j > 10(r_{j+1} + 2 s_{j+1})$, $r_{I_m} - 2 s_{I_m}>0$ and $r_1 < \mathrm{inj}_M/2$, $V$ is stable in at least one of the annuli.}
        \end{lemma}
        \begin{proof}
            \myhl{Suppose not, and then for each $V$ in $\mathbf{C}(S)$, there is $p^V \in M$ such that there exist $5^m$ concentric annuli $A(p^V, r^V_j - s^V_j, r^V_j + s^V_j)$ with the conditions mentioned above and $V$ is not almost minimizing in any one of the annuli.}

            \myhl{As explained in Remark \ref{rmk:CTS}, one could use a family of discrete sweepouts to approximate $S$. Roughly speaking, for each $\Phi_i \in S$ with $X_i = \mathrm{dmn}(\Phi_i)$ where $X_i \subset I(m, n_i)$ is a subcomplex, we can take $N_i > n_i$ large enough and define }
            \begin{equation}
                \psi_i: X_i \cap I(m, N_i)_0 \rightarrow \mathcal{Z}_n(M; \mathbb{Z}_2),    
            \end{equation}
            \myhl{which approximates $\Phi_i$. Note that the discrete sweepouts $S_\# = \set{\psi_i}$ satisfy that $\mathbf{C}(S_\#) = \mathbf{C}(S)$ and the Almgren extension $\Psi_i$ of $\psi_i$ is homotopic to $\Phi_i$ in the flat topology. Interested readers might refer to the proof for Theorem 3.8 in \cite{marques_morse_2016} for more details. Although they require that each $\Phi_i$ is continuous with respect to the $\mathbf{F}$ norm (See the definition in \cite{pitts_existence_1981} 2.1(20)), the same proof works in our setup as well. }

            \myhl{For convenience, in the following, we may assume that $X_i$ itself is a subcomplex of $I(m, N_i)$, since one could always refine $X_i$ in a canonical way. }
        
            \myhl{Now following the original proof of Theorem 4.10 in \cite{pitts_existence_1981}, for $i$ large enough, we can assign to each face $\sigma \in X_i(N_i)$ a set $A(\sigma)=\set{A(p^{V_\sigma}, r^{V_\sigma}_j - s^{V_\sigma}_j, r^{V_\sigma}_j + s^{V_\sigma}_j)}_{j = 1, \cdots, 5^m}$ for some $V_\sigma$ associated to $\sigma$, where $V_\sigma$ is one of the $\set{V_1,\dots, V_\nu}$ therein. By Lemma \ref{lem:comb}, we could define $\alpha$ on $X_i(N_i)$. Therefore, the existsence of a homotopic family of discrete sweepouts $S^*_\# = \set{\psi^*_i}$ such that $\mathbf{L}(S^*_\#) < \mathbf{L}(S_\#) = \mathbf{L}(S)$ is just verbatim.}

            \myhl{The Almgren extension $\Psi^*_i$ of $\psi^*_i$ is homotopic to $\Psi_i$ and thus to $\Phi_i$ which implies that $\set{\Psi^*_i} \subset \Pi$. As long as each $N_i$ is chosen large enough, $\mathbf{L}(\set{\Psi^*_i}) = \mathbf{L}(S^*_\#) < \mathbf{L}(S)$, which gives a contradiction to the choice of $S$. Hence, the conclusion holds.}
        \end{proof}

        By applying \myhl{Theorem 4 in \cite{schoen_regularity_1981}}, the varifold obtained in the lemma above is indeed a singular minimal hypersurface with optimal regularity. 

        \begin{definition}
            We define \textbf{Almgren-Pitts Realization} of the $m$-parameter homotopy family $\Pi$, denoted by $\mathcal{APR}(\Pi)$, to be the nonempty set of varifolds $V$ satisfying
            \begin{itemize}
                \item $\|V\|(M) = \mathbf{L}(\Pi)$\myhl{;}
                \item $V$ is a singular minimal hypersurface with optimal regularity\myhl{;}
                \item $V$ has property $(m)$.
            \end{itemize}
        \end{definition}

        Now, we can show that $\mathcal{APR}$ has following compactness property.

        \begin{proposition}[Compactness of $\mathcal{APR}$]\label{prop:CPT}
            Given a sequence of $m$-parameter homotopy families $\Pi_i$ with $0 < \inf \mathbf{L}(\Pi_i) \leq \sup \mathbf{L}(\Pi_i) < \infty$ \myhl{and} $V_i \in \mathcal{APR}(\Pi_i)$, there is a subsequence (still denoted using index $i$) $\Pi_i$ such that
            \begin{equation}
                 \lim_{i\rightarrow \infty}\mathbf{L}(\Pi_i) = L \in \mathbb{R}^+,
            \end{equation}
            and 
            \begin{equation}
                V_i \rightharpoonup V,
            \end{equation}
            in the varifold sense. Moreover, $\mathrm{spt}(V)$ is also a singular minimal hypersurface with optimal regularity and property $(m)$.
        \end{proposition}
        \begin{proof}
            By Allard compactness, we only need to check that $V$ has property $(m)$ and optimal regularity.

            Firstly, we show that $V$ has property $(m)$.

            Suppose not, and then we can find a set of $I_m$ concentric annuli $\set{A(p, r_j - s_j, r_j + s_j)}$ with the conditions mentioned in property $(m)$, such that $V$ is not stable in any annulus. By Proposition \ref{prop:CPT0}, we know that for each $j$, there is a positive integer $i_j>0$ such that when $i > i_j$, $V_i$ is not stable in each annulus. If we take $i_0 = \max_{j=1, \dots, I_m}\set{i_j}$, then $V_i$ is not stable in any one of $\set{A(p, r_j - s_j, r_j + s_j)}$ provided that $i > i_0$.

            This contradicts to the definition of $\mathcal{APR}(\Pi_i)$ that $V_i$ has property $(m)$.

            Next, to show the optimal regularity of $V$, note that the same argument above together with Proposition \ref{prop:CPT0} also implies that for any $p \in M$ and any $I_m$ concentric annuli $\set{A(p, r_j - s_j, r_j + s_j)}$ with the same properties mentioned above, $V$ is stable and of optimal regularity inside at least one of them. Thus, there exists a constant $r_p \in (0, \mathrm{inj}_M/2)$ depending only on $p$ such that $V$ is stable and of optimal regularity in $A(p, s, r_p)$ for any $s \in (0, r_p)$ (See Theorem 4.10 in \cite{pitts_existence_1981}). 

            By Theorem 3.1 in \cite{wickramasekera_general_2014} and the Remark (3) before it, we know that $V$ is of optimal regularity in the open ball $B(p, r_p)$. Since $M$ is compact, taking a finite open cover, one can easily show that $V$ is of optimal regularity in $M$.
        \end{proof}

    \subsection{Almgren-Pitts Realizations of $(p,m)$-width and $p$-width and Their Compactness}

        In \cite{xu_p;m-width_2016}, G. Xu defines $(p,m)$-width to be
        \begin{equation}
            \omega_{p,m}(M,g) = \inf_{\Phi\in \mathcal{P}_{p,m}}\max_{x \in \mathrm{dmn}(\Phi)} \mathbf{M}(\Phi(x)),
        \end{equation}
        where $\mathcal{P}_{p,m}$ is the set of mass-concentration-free sweepouts from a subcomplex of some $I(m, k)$ into $\mathcal{Z}_n(M; \mathbb{Z}_2)$ detecting $\bar\lambda^p$. And he also proved that when $m\geq 2p+1$, $\mathcal{P}_{p,m} \neq \emptyset$ and the $(p,m)$-width can be realized by a singular minimal hypersurface with optimal regularity.

        Note that the only difference between $p$-width and $(p,m)$-width is the domain of the sweepouts. Since $\mathcal{P}_{p,m}$ is \myhl{an} $m$-parameter homotopy family, the realization of $(p,m)$-width is just a corollary of the compactness property (Proposition \ref{prop:CPT}).
        
        \begin{corollary}[\cite{xu_p;m-width_2016}, Theorem 1.12]
            For $m\geq 2p+1$, there is a varifold $V$ such that $\|V\|(M) = \omega_{p,m}(M)$ and $\mathrm{spt}(V)$ is a singular minimal hypersurface with optimal regularity.
        \end{corollary}

        It is obvious that
        \begin{equation}
            \omega_{p,2p+1}(M) \geq \omega_{p, 2p+2}(M)\geq \dots \geq \omega_{p,m}(M) \geq \dots \geq \omega_p(M).
        \end{equation}
        
        The question is whether for some $m \geq 2p+1$, we can have the equality between $\omega_{p,m}(M)$ and $\omega_p(M)$. Here, we would like to confirm this by a simple argument.

        \begin{proposition}
            $\omega_{p, 2p+1}(M,g) = \omega_p(M,g)$.
        \end{proposition}
        \begin{proof}
            Given a min-max sequence $S = \set{\Phi_i}_{i\in \mathbb{N}} \subset \mathcal{P}_p$ for $\omega_p$, denote $X^{(p)}_i$ as the $p$-dimensional skeleton of $X_i = \mathrm{dmn}(\Phi_i)$, and following the proof in Proposition 2.2 in \cite{irie_density_2018}, we have
            \begin{equation}
                H^l(X_i, X^{(p)}_i; \mathbb{Z}_2) = 0\,\, (l \leq p),
            \end{equation}
            and the exact sequence
            \begin{equation}
                \cdots \rightarrow H^{p}(X_i, X^{(p)}_i; \mathbb{Z}_2) \rightarrow H^{p}(X_i; \mathbb{Z}_2) \rightarrow H^{p}(X^{(p)}_i; \mathbb{Z}_2) \rightarrow H^{p+1}(X_i, X^{(p)}_i; \mathbb{Z}_2)\rightarrow \cdots
            \end{equation}
            The pullback map from $H^p(X_i; \mathbb{Z}_2)$ to $H^p(X^{(p)}_i;\mathbb{Z}_2)$ is injective so $\Phi_i|_{X^{(p)}_i}\in \mathcal{P}_p$. 

            Since
            \begin{equation}
                \omega_p(M,g) \leq \limsup_{i\rightarrow \infty}\set{\mathbf{M}(\Phi_i(x))|x \in X^{(p)}_i} \leq \limsup_{i\rightarrow \infty}\set{\mathbf{M}(\Phi_i(x))|x \in X_i} = \omega_p(M,g),
            \end{equation}
            we also have that
            \begin{equation}
                \limsup_{i\rightarrow \infty}\set{\mathbf{M}(\Phi_i(x))|x \in X^{(p)}_i} = \omega_p(M,g).
            \end{equation}
            As a consequence, we may assume that in the min-max sequence $S$, $\mathrm{dmn}(\Phi_i) = X^{(p)}_i$.

            \myhl{Now that $X^{(p)}_i$ is a finite $p$-dimensional simplicial complex, and thus is homeomorphic to the support of some cubical subcomplex of some $N$-dimensional cube $I^N$ (Chapter 4 \cite{buchstaber_torus_2002}). And we can take a canonical projection (a closed map) mapping $X^{(p)}_i$ into $I^{2p + 1}$ within $1/2$ of the interior. By the general position theorem for maps, Theorem 5.3 in \cite{rourke_introduction_1982}, with $P_0 = \emptyset$ and $P = X^{(p)}_i$ therein, there exists a piecewise-linear embedding map from $X^{(p)}_i$ to some triangulation $T$ of $I^{2p + 1}$ whose image has a distance, say, at least $1/6$ to $\partial I^{2p + 1}$.}
            
            Next, we can ``thicken'' $X^{(p)}_i$ in $I^{2p+1}$ to obtain a subcomplex $Y_i$ of some ${I(2p+1, k)}$ such that $X_i^{(p)}$ is a retract from $Y_i$, which induces a map $\Psi_i:Y_i \rightarrow \mathcal{Z}_n(M; \mathbb{Z}_2)$ by $\Psi_i(y) := \Phi_i(r(y))$. Moreover, $(\Psi_i)^*(\bar\lambda^p) = r^*\comp (\Phi_i)^*(\bar\lambda^p) \neq 0$, since $i^* \comp r^* = id^*$ implies that $r^*$ is an injection. It is also easy to see that $\max_{y \in Y_i}\set{\mathbf{M}(\Psi_i(y))} = \max_{x \in X_i}\set{\mathbf{M}(\Phi_i(x))}$.

            \myhl{Indeed, there exists a regular neighborhood $Z_i$ of $X^{(p)}_i$ in $T'$, a refinement of $T$ (\cite{rourke_introduction_1982} P.33). Moreover, by Corollary 3.30 in \cite{rourke_introduction_1982}, we also know that $X^{(p)}_i$ is a deformation retract of $Z_i$, and thus, we have a retraction $r: Z_i \rightarrow X^{(p)}_i$. Now, since $d = \mathrm{dist}(X^{(p)}_i, \partial Z_i)>0$, we can take a large enough interger $k = k(p, d)$ such that the subcomplex $Y_i \subset I(2p+1, k)$ with} 
            \begin{equation}
                Y_i(k) = \set{\alpha \in I(2p+1, k)| \exists \beta \in I(2p+1, k),\,\,s.t. \,\, \alpha<\beta,\,\, \beta\cap X^{(p)}_i \neq \emptyset},
            \end{equation}
            \myhl{satisfies that $X^{(p)}_i \subset Y_i \subset Z_i$. Apparently, $r|_{Y_i}$ is also a retraction from $Y_i$ to $X^{(p)}_i$, which confirms our assertion.}

            \myhl{In summary, $\omega_p(M,g)$ can be achieved by the sequence $\set{\Psi_i}\subset \mathcal{P}_{p,2p+1}$, so we have $\omega_p(M,g) \geq \omega_{p, 2p+1}(M,g)$ and hence $\omega_p(M,g) = \omega_{p, 2p+1}(M,g)$.}
        \end{proof}

        \begin{corollary}[Realization of $p$-width]
            Each $p$-width can be realized by a singular minimal hypersurface with optimal regularity and moreover, with property $(2p+1)$.
        \end{corollary}
        \begin{remark}
            When $2\leq n \leq 6$, this has been proved (\cite{irie_density_2018}, Proposition 2.2), where they used the upper bound of Morse index of minimal hypersurfaces from min-max construction \cite{marques_morse_2016}. However, when $n\geq 7$, without an adequate alternative of bumpy metrics defined in the singular setting, the technique in \cite{marques_morse_2016} to make a critical sequence bypass all minimal hypersurface with large Morse index could not be applied directly. Thus, the compactness result using \cite{sharp_compactness_2017} is still open.

            \myhl{Even worse, up to the author's knowledge, it is still open whether the minimal hypersurfaces from min-max construction has finite Morse index.}
        \end{remark}

        \begin{definition}
            We define the \textbf{Almgren-Pitts realizations of $p$-width} $\mathcal{APR}_p(M,g)$ to be $\mathcal{APR}(\mathcal{P}_{p, 2p+1})$.
        \end{definition}

        Now we have the following compactness for varying metrics.
        \begin{proposition}[Compactness of $\mathcal{APR}_p$ for Varying Metrics]\label{prop:CPTvary}
            Given a smooth closed manifold $M$ and $C^\infty$ metrics $\set{g_i}$ and $g$ such that
            \begin{equation}
                g_i \xrightarrow{C^3} g,
            \end{equation}
            and $V_i \in \mathcal{APR}_p(M, g_i)$ for some $p > 0$, there is a subsequence of $\set{V_i}$ (still denoted by $\set{V_i}$) such that 
            \begin{equation}
                V_i \rightharpoonup V,
            \end{equation}
            in the varifold sense. Moreover, $V \in \mathcal{APR}_p(M,g)$.
        \end{proposition}
        \begin{proof}
            From the continuity of $p$-width (\cite{irie_density_2018}, Lemma 2.1), we have that $\|V\|(M) = \omega_p(M,g)$. With Remark \ref{rmk:CPTvary}, the proof that $V$ has optimal regularity and property $(2p+1)$ is simply verbatim of the proof of Proposition \ref{prop:CPT}.
        \end{proof}

\section{Proof of Main Theorem}

    For any open subset $U\subset M$, we define
    \begin{equation}
         \mathcal{M}_{U,p} := \set{g \in \Gamma_\infty(M)| \forall V \in \mathcal{APR}_p(M,g),\, \|V\|(U) > 0},
    \end{equation}
    where $\Gamma_\infty(M)$ is the set of all smooth Riemannian metrics on $M$ and let
    \begin{equation}
        \mathcal{M}_U := \bigcup^\infty_{p = 1} \mathcal{M}_{U, p}.
    \end{equation}

    \begin{proposition}
        $\mathcal{M}_{U,p}$ is an open subset of $\Gamma_\infty(M)$ for any open subset $U$, and so is $\mathcal{M}_U$.
    \end{proposition}
    \begin{proof}
        Given $g_0 \in \mathcal{M}_{U,p}$, we would like to show that there is an $\delta > 0$ such that $B_\delta(g_0, C^3)\cap \Gamma_\infty(M) \subset \mathcal{M}_{U,p}$. 

        Suppose not, there will be a sequence $g_i \in \Gamma_\infty(M)$ such that $g_i \xrightarrow{C^3} g_0$ but $g_i \notin \mathcal{M}_{U, p}$. Therefore, we can choose a sequence $\set{V_i}$ such that $V_i \in \mathcal{APR}_p(M,g_i)$ but $V_i(U) = 0$. From Proposition \ref{prop:CPTvary}, up to a subsequence,
        \begin{equation}
            V_i \rightharpoonup V,
        \end{equation}
        where $V \in \mathcal{APR}_p(M,g_0)$.

        Since $U$ is open, $\|V\|(U) \leq \lim_{i\rightarrow \infty}\|V_i\|(U) = 0$ which gives a contradiction.
    \end{proof}

    Now we define $\mathcal{M}_f$ to be the set of metrics on $M$ where there are only finitely many singular minimal hypersurfaces of optimal regularity w.r.t. that metric.

    \begin{lemma}[Key Lemma]\label{lemma:KEY}
        For any open subset $\mathcal{O}$ of $\Gamma_\infty$, if $\mathcal{M}_f$ is dense in $\mathcal{O}$, then $\mathcal{M}_{U}$ is both open and dense in $\mathcal{O}$ as well. Thus, $\mathcal{M}_f\cap \mathcal{O}$ is a meagre set inside $\mathcal{O}$.
    \end{lemma}
    \begin{remark}
        As we will see in the proof, $\mathcal{M}_f$ plays the same role as bumpy metrics in the proof of Proposition 3.1 in \cite{irie_density_2018}.
    \end{remark}
    \begin{proof}
        Fix $U$ as an open subset of $M$. For any $g$ in $\mathcal{O}$, from the denseness of $\mathcal{M}_f$, there is a $g'\in \mathcal{M}_f$ such that $g'$ is arbitrarily close to $g$. Now, if $g' \in \mathcal{M}_{U}$ then we are done.

        Suppose that $g' \notin \mathcal{M}_U$. We can follow the proof of Proposition 3.1 in \cite{irie_density_2018}, since now the set
        \begin{equation}
            \begin{aligned}
                \mathcal{C} = &\set{\sum^N_{j = 1} m_j \mathrm{vol}_{g'}(\Sigma_j): N \in \mathbb{N}, \set{m_j}^N_{j=1}\subset \mathbb{N}, \set{\Sigma_j}^N_{j=1} \\
                 &\text{ singular minimal hypersurfaces with optimal regularity}},
            \end{aligned}
        \end{equation}
        is countable and thus has empty interior.

        Let $h$ be a smooth nonnegative function with $\mathrm{spt}(h)\subset U$ and $h(x) > 0$ for some $x \in U$. Let $g'(t) = (1 + th) g'$. Since $\mathcal{O}$ is open, there is a $t_0 > 0$ such that $\set{g'(t)|0\leq t \leq t_0}\subset \mathcal{O}$. Moreover, using the same argument in Proposition 3.1 in \cite{irie_density_2018}, there exists a $t_1\in(0, t_0)$ arbitrarily small and $p = p(t_0) \in \mathbb{N}$ such that $\omega_p(M, g'(t_1)) > \omega_p(M, g')$ and $\omega_p(M, g'(t_1))\notin \mathcal{C}$. Now it suffices to show that $g'(t_1)\in \mathcal{M}_{U,p}$.

        Suppose not, we can find $V \subset \mathcal{APR}_p(M, g'(t_1))$ such that $\|V\|(U) = 0$. Note that $g'(t_1) = g'$ outside $U$ so we have 
        \begin{equation}
             \|V\|(M) = \sum^N_{j = 1} m_j \mathrm{vol}_{g'(t_1)}(\Sigma_j) = \sum^N_{j = 1} m_j \mathrm{vol}_{g'}(\Sigma_j)\in \mathcal{C},
        \end{equation}
        where $\set{\Sigma_j}$ is a finite set of singular minimal hypersurfaces with optimal regularity with respect to both $g'(t_1)$ and $g'$. This gives a contradiction.

        Let $\set{U_i}$ be a countable basis of $M$, then $\mathcal{M} = \bigcap\mathcal{M}_{U_i}$ is of second Baire category in $\mathcal{O}$ so $\mathcal{M}_f \cap \mathcal{O}\subset \mathcal{M}^C \cap \mathcal{O}$ is a meagre set.
    \end{proof}

    \begin{proof}[Proof of Main Theorem]
        Let $\mathcal{O} = \mathrm{Int}(\overline{\mathcal{M}_f})$ and it is easy to see that $\mathcal{M}_f \subset (\mathcal{M}_f \cap \mathcal{O}) \bigcup \partial(\overline{\mathcal{M}_f})$. From Lemma \ref{lemma:KEY}, we know that $\mathcal{M}_f\cap \mathcal{O}$ is meagre. Since $\partial(\overline{\mathcal{M}_f})$ is nowhere dense, $\mathcal{M}_f$ is also meagre.
    \end{proof}

    \begin{remark}
        In the Key Lemma \ref{lemma:KEY}, we only use the fact that $\mathcal{C}$ is a set with empty interior. Thus, if $\mathcal{M}_{ei}$ is the set of metrics where $\mathcal{C}$ has empty interior and $\mathcal{O}_{ei} = \mathrm{Int}\overline{\mathcal{M}_{ei}}$, we also have that $\mathcal{M}$ is of second Baire category in $\mathcal{O}_{ei}$. As a consequence, we have the denseness of singular minimal hypersurfaces with optimal regularity in generic metrics inside $\mathcal{O}_{ei}$.
    \end{remark}

    \begin{corollary}\label{cor:end}
        For a closed manifold $M^{n+1}(2 \leq n \leq 6)$ with a $C^\infty$-generic metric $g$, the union of minimal hypersurfaces in $\mathcal{APR}(M, g):= \bigcup^\infty_{p = 1} \mathcal{APR}_p(M, g)$, i.e., the minimal hypersurfaces realizing min-max widths, is a dense subset of $M$.
    \end{corollary}

    \begin{proof}
        \myhl{By Theorem 2.7 in \cite{white_bumpy_2017} (an analogue could be refered to Theorem 9 in \cite{ambrozio_compactness_2017-1}), the set of bumpy metrics is generic in $\Gamma_\infty(M)$. It follows from Sharp's compactness theorem (\cite{sharp_compactness_2017} Theorem 2.3 and Remark 2.4) that the bumpy metric belongs to $\mathcal{M}_{ei}$ and therefore, $\mathcal{M}_{ei}$ is also generic and thus dense in $\Gamma_\infty(M)$. From the remark above, we know that the Key Lemma leads to the conclusion.}
    \end{proof}

    Parallel to the existence of bumpy metrics when $2 \leq n \leq 6$, we have the following conjecture.

    \begin{conjecture}\label{conj:empty_interior}
        $\mathcal{M}_{ei}$ is dense in $\Gamma_\infty(M)$. 
    \end{conjecture}

    \begin{corollary}
        If Conjecture \ref{conj:empty_interior} above holds, singular minimal hypersurfaces with optimal regularity in $M$ with generic metrics are dense.
    \end{corollary}

    In particular, let $\mathcal{M}_c$ be the set of the metrics where there are only countably many singular minimal hypersurfaces of optimal regularity w.r.t. that metric and then the following conjecture would imply Conjecture \ref{conj:empty_interior}.

    \begin{conjecture}\label{conj:countable}
        $\mathcal{M}_c$ is dense in $\Gamma_\infty(M)$. 
    \end{conjecture}

    Morally speaking, Conjecture \ref{conj:countable} can even lead to upper bounds of Morse index following the techniques in \cite{marques_morse_2016}.
\bibliography{reference}
\end{document}